\documentclass[review]{elsarticle}
\usepackage{tikz, graphicx}
\usepackage{latexsym}
\usepackage{amsmath,amsthm,amssymb}%,yfonts,pxfonts,pifont,eufrak, bbm}
\usepackage{amsthm}
\usepackage{epstopdf}
\usepackage{lineno,hyperref}
\usepackage[left=2cm,top=1.25cm,right=2cm]{geometry}
\modulolinenumbers[5]

\usepackage{graphics}
\usepackage{graphicx}
\usepackage{color}
\usepackage{amsmath,amsthm,amsfonts,amssymb}%,yfonts,pxfonts,pifont,eufrak, bbm}
\usepackage{epstopdf}
\usepackage{lineno,hyperref}
\modulolinenumbers[5]

\newtheorem{lemma}{Lemma}

\newtheorem{theorem}{Theorem}
\numberwithin{equation}{section}

\begin{document}
\journal{Results in Mathematics}
\begin{frontmatter}
\title{ Approximation by Modified Jain-Baskakov Operators}

\author[label1,label2,label*]{ Vishnu Narayan Mishra}
\ead{vishnunarayanmishra@gmail.com}

\author[label1]{Preeti Sharma}
\ead{preeti.iitan@gmail.com}

\author[label3]{ Marius Birou}
\ead{ marius.birou@math.utcluj.ro}

\address[label1]{Department of Applied Mathematics \& Humanities,
Sardar Vallabhbhai National Institute of Technology, Ichchhanath Mahadev Dumas Road, Surat -395 007 (Gujarat), India}

\address[label2]{L. 1627 Awadh Puri Colony Beniganj, Phase -III, Opposite - Industrial Training Institute (I.T.I.), Ayodhya Main Road Faizabad-224 001, (Uttar Pradesh)}

\address[label3]{Mathematics Department, Technical University of Cluj Napoca, Cluj Napoca, Romania
\\\vspace{0.5cm} \text{\textsf{Dedicated to the 66-th Happy Anniversary of our great Master and friend academician Prof. Dr. Heiner Gonska }}}

\fntext[label*]{Corresponding author}

\begin{abstract}
In the present paper we discuss the approximation properties of Jain-Baskakov operators with parameter $c$. The present paper deals with the modified forms of the Baskakov basis functions. We establish some direct results, which include the asymptotic formula and error estimation in terms of the modulus of continuity and weighted approximation.

\end{abstract}

\end{frontmatter}

\section{Introduction}
\indent In 1972, G.C. Jain \cite{jain1972approximation} introduced and studied the following class of positive linear operators
\begin{equation}\label{jaineq1}
P_n^{[\beta]}(f,x) =  \sum_{v=0}^{\infty} \omega_{\beta}(v,n x) f\left(\frac{v}{n}\right),
\end{equation}
where

\begin{equation}\label{jaineq2}
\omega_{\beta}(v,nx) = nx(nx+v\beta)^{v-1}\frac{e^{-(nx+v\beta)}}{v!},
\end{equation}
and $\beta\in [0,1)$, $f\in C(\mathbb{R}_{+})$. In the particular
case $\beta=0$, $P_n^{0}$, $n\in \mathbb{N}$, turns into
well-known Szasz-Mirakjan operators
\cite{szasz1950generalization}. A Kantorovich-type extension of
$P_n^{[\beta]}$ was given in \cite{umar1985approximation}. In the
last few decades several authors contributed to the approximation
theory, fractional differential and integral equations have
recently been applied in various areas of engineering,
mathematics, physics and bio-engineering and other applied
sciences (see \cite{Dubey2008}, \cite{Rassias},
\cite{tarabie2012jain}, \cite{RASA2009}, \cite{RASA2007},
\cite{RASA2009Gonska}, \cite{RASA2014} and references therein).

Very recently Patel and Mishra \cite{patel2015jain} introduced
modified Jain-Baskakov operators as follows
\begin{eqnarray}\label{jaineq3}
\mathcal{D}_{n,c}^{\beta}(f,x)&=& \frac{(n-c)}{c}\sum_{v=1}^{\infty} \omega_{\beta}(v,n x)\int_0^{\infty} p_{n,v-1,c}(t)f(t) dt + e^{-nx}f(0),
\end{eqnarray}
where $\displaystyle p_{n,v-1,c} (t) = c \frac{\Gamma(n/c+v-1)}{\Gamma(v)\Gamma(n/c)}\cdot \frac{(ct)^{v-1}}{(1+ct)^{n/c+v-1}}$ and $\omega_{\beta}(v,nx)$ are as defined in \eqref{jaineq2}.\\
As a special case, i.e., $c=1$, the operators (\ref{jaineq3})
reduce to the Jain-Baskakov operators which is defined in
\cite{patel2015jain}. The operators given by (\ref{jaineq3}) have
different approximation properties than the operators defined for
$c=1$.

The present paper deals with the study of some direct results in
ordinary approximation for the modified Jain-Baskakov operators
with parameter $c$. We study weighted approximation theorems, rate
of convergence of these operators.
% Also, approximation properties of King-type and Stancu-type approach for the operators \eqref{10.1.eq1} are mention. At the end, we establish some results on modified Jain-Baskakov operators with parameter $c$.

\section{Auxiliary Results}

\begin{lemma}\label{jainlemma1}(\cite{jain1972approximation}, \cite{Farcas}) For $P_{n}^{[\beta]} (t^m,x),~~m=0,...,4$, we have
$$ P_{n}^{[\beta]} (1,x)= 1,~~~ P_{n}^{[\beta]} (t,x)= \frac{x}{1-\beta},~~~ P_{n}^{[\beta]} (t^2,x)= \frac{x^2}{(1-\beta)^2}+ \frac{x}{n(1-\beta)^3},$$

\[
P_{n}^{[\beta ]}(t^{3},x)=\frac{x^{3}}{(1-\beta )^{3}}+\frac{3x^{2}}{%
n(1-\beta )^{4}}-\frac{(6\beta ^{4}-6\beta ^{3}-2\beta
-1)x}{n^{2}(1-\beta )^{5}},
\]

\begin{eqnarray*}
P_{n}^{[\beta ]}(t^{4},x) &=&\frac{x^{4}}{(1-\beta )^{4}}+\frac{6x^{3}}{%
n(1-\beta )^{5}}-\frac{(36\beta ^{4}-72\beta ^{3}+36\beta ^{2}-8\beta -7)x^2%
}{n^{2}(1-\beta )^{6}} \\
&&+\frac{(105\beta ^{5}-14\beta ^{4}-2\beta ^{3}+12\beta ^{2}+8\beta +1)x}{%
n^{3}(1-\beta )^{7}}.
\end{eqnarray*}

\end{lemma}

\begin{lemma}\label{jainlemma2}
 For $n>5c>0$ , we have
 $$\mathcal{D}_{n,c}^{\beta}(1,x)=1, ~~\mathcal{D}_{n,c}^{\beta}(t,x)= \frac{n   x}{(n-2c)(1-\beta)},$$
\begin{equation*}
\mathcal{D}_{n,c}^{\beta}(t^2,x)= \frac{n^2}{(n-2c)(n-3c)} \left[
\frac{x^2}{(1-\beta )^2}+ \frac{x \left(2-2 \beta +\beta
^2\right)}{n (1-\beta )^3}\right],
\end{equation*}
\[
\mathcal{D}_{n,c}^{\beta }(t^{3},x)=\frac{n^{2}x^{2}(-(1-\beta
)xn+3(\beta ^{2}-2\beta +2))}{(1-\beta
)^{4}(n-2c)(n-3c)(n-4c)}+o\left( \frac{1}{n}\right),
\]%
\[
\mathcal{D}_{n,c}^{\beta }(t^{4},x)=\frac{n^{3}x^{3}((1-\beta
)xn+6(\beta ^{2}-2\beta +2))}{(1-\beta
)^{5}(n-2c)(n-3c)(n-4c)(n-5c)}+o\left( \frac{1}{n}\right).
\]
\end{lemma}

\begin{proof}
For $n>(j+1)c,$ we have
\begin{eqnarray*}
\int_{0}^{\infty }t^{j}p_{n,v-1,c}(t)dt &=&\frac{c\Gamma
(n/c+v-1)}{\Gamma (v)\Gamma (n/c)}\cdot \frac{\Gamma (v+j)\Gamma
(n/c-j-1)}{c^{j+1}\Gamma
(n/c+v-1)} \\
&=&c\frac{(v+j-1)...v}{(n-c)...(n-(j+1)c)}.
\end{eqnarray*}%
It follows that
\[
\mathcal{D}_{n,c}^{\beta
}(1,x)=\frac{n-c}{c}\sum\limits_{n=1}^{\infty }\omega _{\beta
}(v,nx)\cdot \frac{c}{n-c}+e^{-nx}=P_{n}^{[\beta ]}(1,x),
\]%
\begin{eqnarray*}
\mathcal{D}_{n,c}^{\beta }(t,x)
&=&\frac{n-c}{c}\sum\limits_{n=1}^{\infty }\omega
_{\beta }(v,nx)\cdot \frac{cv}{(n-c)(n-2c)} \\
&=&\frac{n}{n-2c}\sum\limits_{n=1}^{\infty }\omega _{\beta }(v,nx)\frac{v}{n}%
=\frac{n}{n-2c}P_{n}^{[\beta ]}(t,x),
\end{eqnarray*}%
\begin{eqnarray*}
\mathcal{D}_{n,c}^{\beta }(t^{2},x)
&=&\frac{n-c}{c}\sum\limits_{n=1}^{\infty }\omega
_{\beta }(v,nx)\cdot \frac{cv(v+1)}{(n-c)(n-2c)(n-3c)} \\
&=&\frac{n^{2}}{(n-2c)(n-3c)}\sum\limits_{n=1}^{\infty }\omega
_{\beta
}(v,nx)\frac{v^{2}+v}{n^{2}} \\
&=&\frac{n^{2}}{(n-2c)(n-3c)}\left( P_{n}^{[\beta ]}(t^{2},x)+\frac{1}{n}%
P_{n}^{[\beta ]}(t,x)\right),
\end{eqnarray*}%
\begin{eqnarray*}
\mathcal{D}_{n,c}^{\beta }(t^{3},x)
&=&\frac{n-c}{c}\sum\limits_{n=1}^{\infty }\omega
_{\beta }(v,nx)\cdot \frac{cv(v+1)(v+2)}{(n-c)(n-2c)(n-3c)(n-4c)} \\
&=&\frac{n^{3}}{(n-2c)(n-3c)(n-4c)}\sum\limits_{n=1}^{\infty
}\omega _{\beta
}(v,nx)\frac{v^{3}+3v^{2}+2v}{n^{3}} \\
&=&\frac{n^{3}}{(n-2c)(n-3c)(n-4c)}\left( P_{n}^{[\beta ]}(t^{3},x)+\frac{3}{%
n}P_{n}^{[\beta ]}(t^{2},x)+\frac{2}{n^{2}}P_{n}^{[\beta
]}(t,x)\right),
\end{eqnarray*}%
\begin{eqnarray*}
\mathcal{D}_{n,c}^{\beta }(t^{4},x)
&=&\frac{n-c}{c}\sum\limits_{n=1}^{\infty }\omega _{\beta
}(v,nx)\cdot
\frac{cv(v+1)(v+2)(v+3)}{(n-c)(n-2c)(n-3c)(n-4c)(n-5c)}
\\
&=&\frac{n^{4}}{(n-2c)(n-3c)(n-4c)(n-5c)}\sum\limits_{n=1}^{\infty
}\omega
_{\beta }(v,nx)\frac{v^{4}+6v^{3}+11v^{2}+6v}{n^{4}} \\
&=&\frac{n^{4}}{(n-2c)(n-3c)(n-4c)(n-5c)}\times \\
&&\times \left( P_{n}^{[\beta ]}(t^{4},x)+\frac{6}{n}P_{n}^{[\beta
]}(t^{3},x)+\frac{11}{n^{2}}P_{n}^{[\beta ]}(t^{2},x)+\frac{6}{n^{3}}%
P_{n}^{[\beta ]}(t,x)\right).
\end{eqnarray*}%
Using Lemma \ref{jainlemma1} we get the conclusion.
\end{proof}

\begin{lemma}\label{jainlemma3}
For $n>5c>0$, we have
\begin{eqnarray*}
\mu_{1,n,c}^{\beta}&=& \mathcal{D}_{n,c}^{\beta}(t-x,x) = \frac{x(n\beta+ 2c(1-\beta))}{(n-2c)(1-\beta)},\\
\mu_{2,n,c}^{\beta}&=& \mathcal{D}_{n,c}^{\beta}((t-x)^2,x)\\
&=& \frac{x^2 \left(n^2 \beta ^2+ n \left(c+4 c \beta -5 c \beta
^2\right) + 6 c^2-12 c^2 \beta +6 c^2 \beta ^2\right)}{(n-2 c)
(n-3 c) (1-\beta )^2} + \frac{n x \left(2-2 \beta +\beta
^2\right)}{(n-2 c) (n-3 c) (1-\beta )^3},\\
\mu _{4,n,c}^{\beta } &=&\mathcal{D}_{n,c}^{\beta }((t-x)^{4},x) \\
&=&\frac{(1-\beta )n^{3}\beta^{2}x^{4}(2c(3+4\beta -7\beta
^{2})+\beta ^{2}n)+6n^{3}\beta ^{2}x^{2}(\beta ^{2}-2\beta
+2)}{(n-2c)(n-3c)(n-4c)(n-5c)(1-\beta )^{5}}+o\left(
\frac{1}{n}\right).
\end{eqnarray*}
\end{lemma}

\begin{proof}
The proof follows from Lemma \ref{jainlemma2}.
\end{proof}

\section{Direct theorem}
%Next, we establish a direct local approximation theorem for the Modified Jain-Baskakov operators in ordinary approximation.\\
Let the space $C_B[0,\infty)$ of all continuous and bounded functions be endowed with the norm $\|f\| = \sup\{|f(x)| : x\in [0,\infty)\}$. Further let us consider the following K-functional:
\begin{equation}
\mathcal{K}_2(f,\delta) = \inf_{g \in W^2} \{ \|f-g\| + \delta \|g''\|\},
\end{equation}
where $\delta>0$ and $W^2 = \{g\in C_B[0,\infty) : g', g'' \in C_B[0,\infty)\}.$ By the methods as given in \cite{devore1993constructive}, there exists an absolute constant $M >0$ such that
\begin{equation}\label{9.3.eq1}
\mathcal{K}_2(f,\delta) \leq M \omega_2(f,\sqrt{\delta}),
\end{equation}
 where $$ \omega_2(f,\sqrt{\delta}) = \sup_{0< h\leq \sqrt{\delta}} ~~\sup_{x\in [0,\infty)} |f(x+2h) - 2f(x+h) + f(x) |$$
 is the second order modulus of smoothness of $f\in C_B[0,\infty)$.

\begin{theorem}
For $f\in C_B[0,\infty)$ and $n>3c>0$, we have
\begin{eqnarray*}
| \mathcal{D}_{n,c}^{\beta}(f,x) -f(x) | &\leq& \omega\left(f, \mu_{1,n,c}^{\beta}(x)\right)+ M\omega_2\left(f, \sqrt{\left(\mu_{1,n,c}^{\beta}(x)\right)^2 +\mu_{2,n,c}^{\beta}(x)}\right),\end{eqnarray*}
 where $M$ is a positive constant.
 \end{theorem}

\begin{proof}
We are introducing the auxiliary operators as follows
$$\hat{\mathcal{D}}_{n,c}^{\beta}(f,x)  = \mathcal{D}_{n,c}^{\beta}(f,x)- f\left(x+ \frac{(2c-2 c\beta +n \beta)x }{(n -  2 c) (1-\beta )} \right) +f(x).$$
 Let $g\in W_{\infty}^2$ and $x,t \in [0,\infty)$. By Taylor's expansion we have
 $$g(t) = g(x) + (t-x) g'(x) + \int_x^t (t-u)g''(u) du.$$
Applying $\hat{\mathcal{D}}_{n,c}^{\beta}$, we get
 $$\hat{\mathcal{D}}_{n,c}^{\beta}(g,x) - g(x) = g'(x) \hat{\mathcal{D}}_{n,c}^{\beta}((t-x),x) +  \hat{\mathcal{D}}_{n,c}^{\beta}\left(\int_x^t (t-u)g''(u) du,x\right).$$
% Applying Lemma \ref{10.2.lemma1}, we get
Now,
\begin{eqnarray*}
 |\hat{\mathcal{D}}_{n,c}^{\beta}(g,x) - g(x)| &\leq & \hat{\mathcal{D}}_{n,c}^{\beta}\left(\bigg|\int_x^t (t-u)g''(u) du\bigg|,x\right)\\
 &\leq & \mathcal{D}_{n,c}^{\beta}\left((t-x)^2,x\right)\|g''\| \\
 &&+ \bigg|\int_x^{x+\frac{(2c -2 c \beta +n \beta)x }{(n-2c) (1-\beta )}}\left(x+\frac{(2c-2c \beta +n \beta)x }{(n-2c) (1-\beta )}-u\right)g''(u) du\bigg|\\
 &\leq & \left[\mu_{2,n,c}^{\beta}(x) +\left(\frac{(2c-2c \beta +n \beta)x }{(n-2c) (1-\beta )}\right)^2\right]\|g''\|.
 \end{eqnarray*}
 Since, $$|\mathcal{D}_{n,c}^{\beta}(f,x)|\leq (n-c)/c \sum_{v=1}^{\infty} w_{\beta}(v,nx) \int_0^{\infty} p_{n,v-1,c}(t)| f(t)| dt + e^{-nx} |f(0)| \leq  \|f\|,$$

 \begin{eqnarray*}
 |\mathcal{D}_{n,c}^{\beta}(f,x) -f(x)| &\leq& |\hat{\mathcal{D}}_{n,c}^{\beta}(f-g,x)-(f-g)(x)| + |\hat{\mathcal{D}}_{n,c}^{\beta}(g,x)-g(x)|\\
  &&+ \bigg|\left(x+\frac{(2c-2c \beta +n \beta)x }{(n-2c) (1-\beta )}\right)-f(x)\bigg|\\
 &\leq& 2\|f-g\| +\left[ \mu_{2,n,c}^{\beta}(x) +\left(\frac{(2 c-2c \beta +n \beta)x }{(n-2c) (1-\beta )}\right)^2 \right]\|g''\| \\
 &&+ \omega\left( f, \frac{(2c-2 c\beta +n \beta)x }{(n-2c) (1-\beta )} \right).
\end{eqnarray*}
Taking infimum overall $g\in W^2$, we get
\begin{eqnarray*}
|\mathcal{D}_{n,c}^{\beta}(f,x) -f(x)| &\leq & \mathcal{K}_2\left(f, \mu_{2,n,c}^{\beta}(x) +\left(\frac{(2c-2c \beta +n \beta)x }{(n-2c) (1-\beta )}\right)^2 \right) \\ &&+\omega\left( f, \frac{(2c-2c \beta +n \beta)x }{(n-2c) (1-\beta )} \right).
\end{eqnarray*}
In view of \eqref{9.3.eq1}, we have
\begin{eqnarray*}
|\mathcal{D}_{n,c}^{\beta}(f,x) -f(x)| &\leq& M\omega_2\left(f, \sqrt{\mu_{2,n,c}^{\beta}(x) +\left(\frac{(2c-2c \beta +n \beta)x }{(n-2c) (1-\beta )}\right)^2} \right) \\ &&+\omega\left( f, \frac{(2 c-2c \beta +n \beta)x }{(n-2c) (1-\beta )} \right),
\end{eqnarray*}
which proves the Theorem.
\end{proof}

\subsection{Rate of Convergence}
Let $B_{\rho_0}(\mathbb{R}^{+})$ be the set of all functions $f$ defined on $\mathbb{R}^{+}$ satisfying the condition $|f(x)| \leq M_f (1+  x^2)$, where $M_f$ is a constant depending only on $f$. Now, $ C_{\rho_0}(\mathbb{R}^{+})$ can be represented as\\
  $\{ f\in B_{\rho_0}(\mathbb{R}^{+}) : f  \text{ is continuous and }   \frac{f(x)}{\rho_{0}(x)} \text{ is convergent as } x \to \infty \}$. For any positive $a$, by
$$\omega_a(f,\delta) = \sup_{|t-x|\leq \delta}~\sup_{x,t\in [0,a]} |f(t)-f(x)|$$
we denote the usual modulus of continuity of $f$ on the closed interval $[0,a]$. We know that for a function $f\in C_{\rho_0}(\mathbb{R}^{+})$, the modulus of the continuity $\omega_a(f,\delta)$ tends to zero.\\
Now, we give a rate of convergence theorem for the Modified Jain-Baskakov operators.

\begin{theorem}\label{10.3.1.thm1}
Let $f\in C_{\rho_0}(\mathbb{R}^{+})$ and $\omega_{a+1}(f,\delta)$
be its modulus of continuity on the finite interval $[0,a+1]
\subset \mathbb{R^+}$, where $a>0$. Then for $n>3c$,
$$ \|\mathcal{D}_{n,c}^{\beta}(f)-f\|_{C[0,a]} \leq 6M_f(1+a^2)\mu_{2,n,c}^{\beta}(x) +2 \omega_{a+1}\left(f,\sqrt{\mu_{2,n,c}^{\beta}(x)}\right).$$
\end{theorem}

\begin{proof}
For $x\in [0,a]$ and $t>a+1$, since $t-x>1$, we have
\begin{eqnarray}\label{9.4.eq1}
|f(t)-f(x)| &\leq& M_f(2+x^2+t^2)\nonumber\\
&\leq & M_f\left(2+ 3x^2 + 2(t-x)^2\right)\nonumber\\
&\leq & 6M_f(1+a^2)(t-x)^2.
\end{eqnarray}
For $x\in [0,a]$ and $t\leq a+1$, we have
\begin{equation}\label{9.4.eq2}
|f(x)-f(t)| \leq \omega_{a+1}(f,|t-x|) \leq \left(1+ \frac{|t-x|}{\delta}\right)\omega_{a+1}(f,\delta)
\end{equation}
with $\delta>0$.\\
From \eqref{9.4.eq1} and \eqref{9.4.eq2} we can write
\begin{equation} |f(t)-f(x)|\leq 6M_f(1+a^2)(t-x)^2 + \left(1+ \frac{|t-x|}{\delta}\right)\omega_{a+1}(f,\delta),\end{equation}
for $x\in [0,a]$ and $t\geq 0$. Thus
\begin{eqnarray*}
|\mathcal{D}_{n,c}^{\beta}(f,x)-f(x)| &\leq& \mathcal{D}_{n,c}^{\beta}\left(|f(t)-f(x)|,x\right)\\
&\leq & 6M_f(1+a^2) \mathcal{D}_{n,c}^{\beta}((t-x)^2,x)\\
&&+ \omega_{a+1}(f,\delta)\left(1+ \frac{1}{\delta}\mathcal{D}_{n,c}^{\beta}((t-x)^2,x)^{1/2}\right).
\end{eqnarray*}
Hence, by Schwarz's inequality and Lemma \ref{jainlemma3}, for $x\in [0,a]$
\begin{eqnarray*}
|\mathcal{D}_{n,c}^{\beta}(f,x)-f(x)|\leq 6M_f(1+a^2)
\mu_{2,n,c}^{\beta}(x) + \omega_{a+1}(f,\delta)\left(1+
\frac{1}{\delta}\sqrt{\mu_{2,n,c}^{\beta}(x)}\right).\end{eqnarray*}
By taking $\delta= \sqrt{\mu_{2,n,c}^{\beta}(x)}$, we get
$$\|\mathcal{D}_{n,c}^{\beta}(f)-f\|_{C[0,a]} \leq 6M_f(1+a^2)\mu_{2,n,c}^{\beta}(x) +
2\omega_{a+1}\left(f,\sqrt{\mu_{2,n,c}^{\beta}(x)}\right).$$ This
completes the proof of Theorem.
\end{proof}

\subsection{Weighted approximation}
Now we shall discuss the weighted approximation theorem, where the approximation formula holds true on the interval $\mathbb{R^+}$.
\begin{theorem}\label{6.5.thm5.1}
If $f\in C_{\rho_{0}}(\mathbb{R}^{+})$, $\displaystyle \lim_{n\to
\infty} \beta_n =0$ and $n>3c$. Then,  $$\lim_{n\to \infty}
\|\mathcal{D}_{n,c}^{\beta_n}(f)-f\|_{\rho_0} = 0.$$
\end{theorem}
\begin{proof}
Using the theorem in \cite{gadjiev1976theorems} we see that it is sufficient to verify the following three conditions
\begin{equation}\label{9.5.eq1}
\lim_{n\to \infty} \|\mathcal{D}_{n,c}^{\beta_n}(t^r,x)-x^r\|_{\rho_0} = 0, ~~ r=0,1,2.
\end{equation}
Since $\mathcal{D}_{n,c}^{\beta_n}(1,x)=1$, the first condition of \eqref{9.5.eq1} is satisfied for $r=0.$\\
Now, for $n>2 c$
\begin{eqnarray*}
\|\mathcal{D}_{n,c}^{\beta_n}(t,x)-x\|_{\rho_0}&=& \sup_{x\in [0,\infty)} \frac{|\mathcal{D}_{n,c}^{\beta_n}(t,x)-x|}{1+x^2}\\
&\leq& \bigg|\frac{n}{(n-2c)(1-\beta_n)}-1\bigg|\sup_{x\in [0,\infty)} \frac{x}{1+x^2} \\
&\leq& \frac{2 c}{(n-2c)} + \frac{n \beta_n }{(n-2c) (1-\beta_n)}
\end{eqnarray*}
which implies $\|\mathcal{D}_{n,c}^{\beta_n}(t,x)-x\|_{\rho_0}=0$ as $n\to \infty$ with $\beta_n \to 0$.\\
Similarly, we can write for $n>3c$
\begin{eqnarray*}
\|\mathcal{D}_{n,c}^{\beta_n}(t^2,x)-x^2\|_{\rho_0}&=& \sup_{x\in [0,\infty)} \frac{|\mathcal{D}_{n,c}^{\beta_n}(t^2,x)-x^2|}{1+x^2}\\
&\leq& \bigg| \frac{n^2}{(n-2 c)(n-3 c)(1-\beta_n )^2}-1\bigg|\sup_{x\in [0,\infty)}\frac{x^2}{1+x^2}\\
&&+\bigg|\frac{n\left(2-2 \beta_n +\beta_n ^2\right)}{(n-2c)(n-3c) (1-\beta_n )^3}\bigg|\sup_{x\in [0,\infty)}\frac{x}{1+x^2}\\
%&\leq& \bigg|\frac{n^2  \beta^2_n}{(n-3c) (n-2c) (1-\beta_n )^2}\bigg|+\bigg|\frac{n(1+5 \beta_n )}{(n-3c) (n-2c) (1-\beta_n )}\bigg|\\
%&&+\bigg|\frac{6}{(n-3c) (n-2c)}  \bigg|+\bigg|\frac{n\left(2-2 \beta_n +\beta_n ^2\right)}{(n-2c)(n-3c) (1-\beta_n )^3}\bigg|,
\end{eqnarray*}
which implies that
$$\lim_{n\to \infty} \|\mathcal{D}_{n,c}^{\beta_n}(t^2,x)-x^2\|_{\rho_0} =0~~ \text{with} ~~\beta_n \to 0. $$
Thus the proof is completed.
\end{proof}

We give the following theorem to approximate all function in $C_{\rho_0}(\mathbb{R}^{+})$. This type of results are given in \cite{gadjiev2003korovkin} for locally integrable functions.
\begin{theorem}
Let $\beta_n\to 0$ as $n \to \infty$. For each $f\in C_{\rho_0}(\mathbb{R}^{+})$ and $\lambda>0$, we have
$$\lim_{n\to \infty}~~ \sup_{x\in [0,\infty)} \frac{|\mathcal{D}_{n,c}^{\beta_n}(f,x) -f(x)|}{(1+x^2)^{1+\lambda}}=0.$$
\end{theorem}
\begin{proof}
For any fixed $x_0>0$,
\begin{eqnarray*}
\sup_{x\in [0,\infty)} \frac{|\mathcal{D}_{n,c}^{\beta_n}(f,x) -f(x)|}{(1+x^2)^{1+\lambda}} &\leq& \sup_{x\leq x_0} \frac{|\mathcal{D}_{n,c}^{\beta_n}(f,x) -f(x)|}{(1+x^2)^{1+\lambda}}+\sup_{x\geq x_0} \frac{|\mathcal{D}_{n,c}^{\beta_n}(f,x) -f(x)|}{(1+x^2)^{1+\lambda}}\\
&\leq & \|\mathcal{D}_{n,c}^{\beta_n} (f) -f\|_{C[0,x_0]}\\
 &&+ \|f\|_{\rho_0} \sup_{x\geq x_0} \frac{|\mathcal{D}_{n,c}^{\beta_n}(1+t^2,x)|}{(1+x^2)^{1+\lambda}} + \sup_{x\geq x_0} \frac{|f(x)|}{(1+x^2)^{1+\lambda}}.
\end{eqnarray*}
The first term of the above inequality tends to zero from Theorem \ref{10.3.1.thm1}. By Lemma \ref{jainlemma2} for any fixed $x_0>0$ it is  easily seen that
$\displaystyle \sup_{x\geq x_0} \frac{|\mathcal{D}_{n,c}^{\beta_n}(1+t^2,x)|}{(1+x^2)^{1+\lambda}} \to 0$ as $n\to \infty$ with $\beta_n \to 0$. We can choose $x_0>0$ so large that the last part of the above inequality can be made small enough.\\
Thus the proof is completed.
\end{proof}

\begin{theorem} Let $b>0$ and $\beta _{n}\in (0,1)$ such that $n\beta _{n}\rightarrow l\in
%TCIMACRO{\U{211d} }%
%BeginExpansion
\mathbb{R}
%EndExpansion
$ as $n\rightarrow \infty .$

Then for every $f\in C[0,b],$ twice differentiable at a fixed
point $x\in (0,b),$ we have
\[
\lim_{n\rightarrow \infty }n(D_{n,c}^{\beta
_{n}}(f,x)-f(x))=x(l+2c)f^{\prime }(x)+\frac{x(2+xc)}{2}f^{\prime
\prime }(x).
\]%
.
\end{theorem}

\begin{proof} From Lemma \ref{jainlemma3} we have%
\[
\lim_{n\rightarrow \infty }n\mu _{1,n,c}^{\beta _{n}}=x(l+2c),
\]%
\[
\lim_{n\rightarrow \infty }n\mu _{2,n,c}^{\beta _{n}}=x(2+xc),
\]%
\[
\lim_{n\rightarrow \infty }n\mu _{4,n,c}^{\beta _{n}}=0.
\]

Using Theorem 2.1 from \cite{Cardenas} we get the conclusion.
\end{proof}

\section{King's Approach}
To make the convergence faster in 2003, King \cite{king2003positive} proposed an approach to modify the classical Bernstein polynomial, so that the sequence preserve test functions $e_0$ and $e_2$. After this approach many researcher contributed in this direction. Motivated by such type of operators, we generalized King-type approach of the operators \eqref{jaineq3} as follows:\\
For $\beta \in [0,1)$, $x\in [0,\infty)$ and $f\in C_{\rho_0}(\mathbb{R}^{+})$, we introduced King-type positive linear operators as
\begin{eqnarray}\label{10.4.eq1}
\mathcal{D}_{n,c}^{*\beta}(f,x)&=& \frac{(n-c)}{c} \sum_{v=1}^{\infty} \omega_{\beta}(v,n r_n(x))\int_0^{\infty} p_{n,v-1}(t)f(t) dt +e^{-nr_n(x)}f(0),
\end{eqnarray}
where $\displaystyle r_n(x) = \frac{(n - 2c)(1-\beta)x}{n}$ and $\omega_{\beta}(v,nx)$ and $p_{n,v-1}(t)$ as same as defined in \eqref{jaineq2} and \eqref{jaineq3}.
\begin{lemma}\label{10.4.lemma2}
The following holds
$$ \mathcal{D}_{n,c}^{*\beta}(1,x)= 1, ~~~~ \mathcal{D}_{n,c}^{*\beta}(t,x)=x,$$
$$ \mathcal{D}_{n,c}^{*\beta}(t^2,x)= \frac{n-2 c}{(n-3c)}  x^2 +  \frac{(2-2\beta+\beta^2)}{(n-3c) (1-\beta )^2} x
,\  n>3c,$$
\[
\mathcal{D}_{n,c}^{\ast \beta }(t^{3},x)=\frac{x^{2}n((1-\beta
^{2})(n-4c)x+3(\beta ^{2}-2\beta +2))}{(1-\beta
)^{2}(n-3c)(n-4c)}+o\left( \frac{1}{n}\right) ,\ n>4c,
\]
\[
\mathcal{D}_{n,c}^{\ast \beta }(t^{4},x)=\frac{n^{2}x^{3}((1-\beta
)^{2}(n-6c)x+6(\beta ^{2}-2\beta +2))}{(1-\beta )^{2}(n-3c)(n-4c)(n-5c)}%
+o\left( \frac{1}{n}\right) ,\ n>5c.
\]
\end{lemma}

\noindent The proof follows from Lemma \ref{jainlemma2} and linearity of $\mathcal{D}_{n,c}^{*\beta}$. We observe that the  operators $\mathcal{D}_{n,c}^{*\beta}$ are preserves linear functions that means, $\mathcal{D}_{n,c}^{*\beta}(a+bt,x) = a + b x$ for constants $a, b$.

\begin{lemma}\label{lema5}
The following holds
$$\mu _{1,n,c}^{\ast \beta }= \mathcal{D}_{n,c}^{*\beta}(t-x,x) =0,$$
$$\mu _{2,n,c}^{\ast \beta }=\mathcal{D}_{n,c}^{*\beta}((t-x)^2,x)
= \frac{ c}{n-3c}  x^2 +  \frac{(2-2\beta+\beta^2)}{(n-3c)
(1-\beta )^2} x ,\  n>3c ,$$
\[
\mu _{4,n,c}^{\ast \beta
}=\mathcal{D}_{n,c}^{*\beta}((t-x)^4,x)=o\left(
\frac{1}{n}\right),\ n>5c.
\]
\end{lemma}
The proof follows from Lemma \ref{10.4.lemma2}.

\begin{theorem}
Let $\mathcal{D}_{n,c}^{*\beta_n}, n > 3c $, be defined as in \eqref{10.4.eq1}, where $\displaystyle \lim_{n\to \infty} \beta_n =0$. For any compact $B\subset \mathbb{R}^{+}$ and for each $f\in C_{\rho_0}(\mathbb{R}^+)$ one has
$$ \lim_{n\to \infty} \mathcal{D}_{n,c}^{*\beta_n}(f,x) = f(x),\text{ uniformly in } x\in B.$$
\end{theorem}
The proof follows from Lemma \ref{10.4.lemma2} and Korovkin theorem \cite{korovkin1953convergence}.

\begin{theorem}\label{10.4.thm2}
 For $f\in C_B[0,\infty)$, $0<\beta<1$ and $n>3 $, we have
 $$| \mathcal{D}_{n,c}^{*\beta}(f,x) -f(x) | \leq C_1\omega_2\left(f, \sqrt{\hat{\mu}_{n,2,c}^\beta}\right),$$
 where $C_1$ is a positive constant.
\end{theorem}

\begin{proof}
Let $g\in W_{\infty}^2$ and $x\in [0,\infty)$. Using Taylor's expansion
$$ g(t) = g(x) + g'(x) (t-x) + \int_x^t (t-u)g''(u) du, ~~t\in [0,\infty)$$
and Lemma \ref{10.4.lemma2}, we have

$$\mathcal{D}_{n,c}^{*\beta}(g,x) -g(x) = \mathcal{D}_{n}^{*\beta}\left( \int_x^t (t-u)g''(u) du\right).$$
Notice that, $\displaystyle \bigg|\int_x^t (t-u)g''(u) du \bigg|\leq (t-x)^2\|g''\|$. Thus
$$|\mathcal{D}_{n,c}^{*\beta}(g,x) -g(x)| \leq \mathcal{D}_{n,c}^{*\beta}((t-x)^2,x) \|g''(x)\| =\mu _{2,n,c}^{\ast \beta } \|g''\|.$$

%\frac{ c}{n-3 c} x^2 + \frac{x \left(2-2 \beta +\beta ^2\right)}{(n- 3 c) (1-\beta )^2}
%From $0<\beta<1$, we have  $0<1-\beta<1$. Therefore
%\begin{eqnarray*}
%|\mathcal{D}_{n}^{*\beta}(g,x) -g(x)| &\leq& \frac{x^2(1-\beta)^2+ x \left(1+ (1-\beta)^2\right)}{(n-3) (1-\beta )^2}\|g''\|\\
%&&\leq \frac{x^2+ 2x}{(n-3) (1-\beta )^2}\|g''\|.
%\end{eqnarray*}
Since $|\mathcal{D}_{n,c}^{*\beta}(f,x)|\leq \|f\|$,

\begin{eqnarray*}
|\mathcal{D}_{n,c}^{*\beta}(f,x)-f(x)| &\leq& |\mathcal{D}_{n,c}^{*\beta}(f-g,x)-(f-g)(x)| +
|\mathcal{D}_{n,c}^{*\beta}(g,x)-g(x)|\\
&\leq & 2\|f-g\| + %\frac{x^2+ 2x}{(n-3) (1-\beta )^2}
\mu _{2,n,c}^{\ast \beta } \|g''\|.
\end{eqnarray*}
Finally taking the infimum   on right side over all $g\in
W_{\infty}^2$ and using \eqref{9.3.eq1}, we get the desired
result.
\end{proof}

Now, we establish the Voronovskaja type asymptotic formula for the operators $\mathcal{D}_{n,c}^{*\beta}$:

\begin{theorem}
Let $b>0$ and $\beta _{n}\in (0,1)$ such that $n\beta
_{n}\rightarrow 0$ as $\rightarrow \infty .$

Then for every $f\in C[0,b],$ twice differentiable at a fixed
point $x\in (0,b),$ we have
$$\lim_{n\rightarrow
\infty}n\big(\mathcal{D}_{n,c}^{*\beta_n}(f,x)-f(x)\big)=x(2+x
c)\frac{f''(x)}{2}.$$
\end{theorem}

\begin{proof}
From Lemma \ref{lema5} we have%
\[
\lim_{n\rightarrow \infty }n\mu _{1,n,c}^{\ast \beta _{n}}=0,
\]%
\[
\lim_{n\rightarrow \infty }n\mu _{2,n,c}^{\ast \beta
_{n}}=x(2+xc),
\]%
\[
\lim_{n\rightarrow \infty }n\mu _{4,n,c}^{\ast \beta _{n}}=0.
\]

Using Theorem 2.1 from \cite{Cardenas} we get the conclusion.
\end{proof}

\textbf{References:}
\bibliographystyle{amsplain}

\end{document}